\pdfoutput=1
\documentclass[12pt]{scrartcl}

\usepackage[scaled=.90]{helvet}
\usepackage{mathptmx}
\usepackage{courier}
\usepackage[round]{natbib}
\usepackage[ngerman,english]{babel}
\usepackage{dsfont}
\usepackage{amsmath}
\usepackage{amssymb}

\usepackage{enumerate, amssymb, amsmath, amsthm}
\newcommand{\N}{\ensuremath{{\mathbb N}}}

\usepackage{tikz} 
\usepackage[T1]{fontenc}
\usepackage{lmodern}
\usepackage[figurewithin=section]{caption}
\usepackage{units}
\usepackage{hyperref}

\newtheorem{satz}{Theorem}[section]

\newtheorem{lem}[satz]{Lemma}

\newcommand{\R}{\ensuremath{{\mathbb R}}}

\newcommand{\Y}{\ensuremath{{\mathcal Y}}}

\begin{document}

\title{Convergence of switching diffusions}

\author{S\"oren Christensen\thanks{Christian-Albrechts-Universit\"at, Mathematisches Seminar, Ludewig-Meyn-Str. 4, 24098 Kiel, Germany, email: \textit{lastname}@math.uni-kiel.de.}\;\;\thanks{University of Hamburg, Department of Mathematics, SPST, Bundesstra\ss e 55, 20146 Hamburg, Germany.}\;, Albrecht Irle\footnotemark[1]}
\date{\today}
\maketitle
\begin{center}
\end{center}

\begin{abstract}
This paper studies the asymptotic behavior of processes with switching. More precisely, the stability under fast switching for diffusion processes and discrete state space Markovian processes is considered. The proofs are based on semimartingale techniques, so that no Markovian assumption for the modulating process is needed. 
\end{abstract}

\textbf{Keywords:} processes with switching, switching diffusions, fast switching, asymptotic stability, semimartingales \vspace{.8cm}

\textbf{Subject Classifications:} 93E03, 60H20, 60H30\\

\section{Introduction}
Stochastic processes with dynamics depending on a further source of randomness have been of interest as well for theoretical reasons as from the point of view of application. Such processes are called processes with switching and usually switching involves an additional Markovian source of randomness with a finite number of states. For diffusion processes $(X_t)_{t\in[0,\infty)}$ given by a stochastic differential equation, the dynamics then additionally depend on a  modulating Markovian process $(Y_t)_{t\in[0,\infty)}$ with finite state space. \cite{Mao} gives an extensive treatment of this subject. If $(X_t)_{t\in[0,\infty)}$ is itself a Markovian process with a discrete state space then the intensity matrix will depend on the modulating process.

In this paper we are interested in the stability under fast switching. This  concerns the behavior of the switching process when the modulating process depends on an additional parameter $\epsilon > 0$ which lets it fluctuate more and more rapidly when $\epsilon$ tends to 0, so we have processes
$(X_t^\epsilon)_{t\in[0,\infty)}$ and $(Y_t^\epsilon)_{t\in[0,\infty)}$.
The question of interest concerns the asymptotic stability, i.e. convergence in a distributional sense, of
$(X_t^\epsilon)_{t\in[0,\infty)}$ as  $\epsilon$ tends to 0 which is by no means obvious as the processes $(Y_t^\epsilon)_{t\in[0,\infty)}$ fluctuate more and more rapidly.

In fast Markovian switching we look at processes
$(Y_t^\epsilon)_{t\in[0,\infty)}$  with intensity matrix $\frac{1}{\epsilon} G$ for a given intensity matrix $G$; in pathwise terms we would look at
$(Y_{t/\epsilon})_{t\in[0,\infty)}$. In \cite{Skorohod} and \cite{Skorohod1} the asymptotic stability in the stochastic differential equation setting was shown, and in \cite{Kabanov,Kabanov1} this stability was derived for conditionally Poisson processes. For the proofs the assumption of Markovian switching was essential and the technical details can be seen as complicated
and technically involved. Note that the processes
$(X_t^\epsilon)_{t\in[0,\infty)}$ themselves are not Markovian so that the usual machinery for showing distributional convergence of Markov processes
cannot be applied directly and has to be adapted.

This note stems from the observation that the processes $(X_t^\epsilon)_{t\in[0,\infty)}$
are semimartingales, so that we may show stability using the convergence theorem for families of semimartingales. As demonstrated here, this can indeed be done. Section \ref{sec:diffusion} is devoted to the case of diffusion processes which is technically more involved relying on some uniform estimates for switching diffusions; Section \ref{sec:discrete} treats discrete state space Markovian processes where the proofs are simpler. The Appendix \ref{appendixa} contains the proof of the analytical Lemma \ref{lem:analytic} which is essential for obtaining the main results.\\
An advantage of the semimartingale approach is that the Markovian assumption for the modulating process is no longer needed, only an  assumption of ergodicity. Furthermore, the proofs turn out to be less complicated than using an approach based Markov theory. 
As switching processes have various applications in financial market modeling, see e.g
\cite{IKLM} or \cite{CIK}, where the modulating process may correspond to macroeconomic influences, this generalization might be of interest in this field. In particular, the findings discussed in \cite{CIK} are based on the results presented in this paper.

\section{Diffusion processes}\label{sec:diffusion}
We consider c\`adl\`ag processes $(X_t)_{t\in[0,\infty)},(Y_t)_{t\in[0,\infty)}$, where $X_t:\Omega\rightarrow I$ for some interval $I\subseteq \R$, $Y_t:\Omega\rightarrow \Y$ for some suitable space $\Y$. Assume that for
\[b:I\times \Y\rightarrow\R,\;\;\sigma:I\times \Y\rightarrow \R\]
 the process $(X_t)_{t\in[0,\infty)}$ fulfills the stochastic differential equation
\[dX_t=b(X_t,Y_t)dt+\sigma(X_t,Y_t)dW_t,\]
where $(W_t)_{t\in[0,\infty)}$ is a Wiener process independent of $(Y_t)_{t\in[0,\infty)}$. Note that the dynamics of the process $X$ depend on the modulating process $Y$, so that we may also write $X=X^Y$.

\begin{lem}\label{lem:moment_estimate}
Assume that there exist $C_1,C_2$ such that
\[\max\{|b(x,y)|,|\sigma(x,y)|\}\leq C_1+C_2|x|\mbox{ for all }x\in I,y\in \Y.\]
Let $q\geq 1, \;\;E|X_0|^q < \infty$. Then for all $T>0$ there exists some constant\ $C_3$, only depending on $q,C_1,C_2,T$ and  $E|X_0|^q$, such that
\[\sup_{t\in[0,T]}E|X_t|^q\leq C_3.\]

\end{lem}

\begin{proof}
This follows immediately from \cite[Lemma 2 and Corollary 6 in Section 2.5]{Krylov}.
\end{proof}

It is important to note that this estimate holds uniformly in all processes $(Y_t)_{t\in[0,\infty)}$ taking values in $\Y$, and  this will also be explicitly stated in the following result:
\begin{lem}\label{lem:unif}
Under the assumptions of Lemma 2.1 let $E|X_0| < \infty$. Then for any $\delta>0$ there exists some $K>0$ such that 
\[P\left(\sup_{t\leq T}|X_t|\geq K\right)\leq \delta\]
uniformly in all $\Y$-valued modulating processes $(Y_t)_{t\in[0,\infty)}$ for $X=X^Y$.
\end{lem}
\begin{proof}
Use $3K$ instead of $K$. Then we see that
\begin{align*}
&P\left(\sup_{t\leq T}\big|X_0+\int_0^t b(X_s,Y_s)ds+\int_0^t\sigma(X_s,Y_s)dW_s\big|\geq 3K\right)\\
\leq&P(|X_0|\geq K)+P\left(\sup_{t\leq T}\big|\int_0^t b(X_s,Y_s)ds\big|\geq K\right)+P\left(\sup_{t\leq T}\big|\int_0^t\sigma(X_s,Y_s)dW_s\big|\geq K\right)
\end{align*}
The first term is trivial, so we start by looking at the second term. Clearly
\begin{align*}
\sup_{t\leq T}\left |\int_0^t b(X_s,Y_s)ds\right|\leq \int_0^T|b(X_s,Y_s)|ds
\leq \int_0^T\left(C_1+C_s|X_s|\right)ds=C_1T+\int_0^TC_2|X_s|ds.
\end{align*}
Now, using Lemma \ref{lem:moment_estimate},
\[E\int_0^TC_2|X_s|ds=\int_0^TC_2E|X_s|ds\leq C_2C_3T,\]
so that the expectation of the second term is bounded by $(C_1+C_2C_3)T$, and the second term is bounded by Markov's inequality by
\[P\left(\sup_{t\leq T}|\int_0^t b(X_s,Y_s)ds|\geq K\right)\leq \frac{(C_1+C_2C_3)T}{K}.\]
Now we consider the third term: The stochastic process $\left(\int_0^t\sigma(X_s,Y_s)dW_s\right)_{t\geq 0}$ is, for another Wiener process $(W'_t)_{t\geq 0}$, equal to $(W'_{\beta(t)})_{t\geq 0}$, where $\beta(t)=\int_0^t\sigma^2(X_s,Y_s)ds$. Hence, for any $\gamma>0$
\begin{align*}
P\left(\sup_{t\leq T}|\int_0^t\sigma(X_s,Y_s)dW_s|\geq K\right)&=P\left(\sup_{t\leq T}|W'_{\beta(t)}|\geq K\right)
=P\left(\sup_{t\leq \beta(T)}|W'_{t}|\geq K\right)\\
&\leq  P\left(\beta(T)\geq \gamma\right)+P\left(\sup_{t\leq \gamma}|W'_{t}|\geq K\right).
\end{align*}
Note that
\[E\beta(T)\leq E\int_0^T(C_1+C_2|X_s|)^2ds=\int_0^TE(C_1+C_2|X_s|)^2ds\leq C_4\]
by Lemma \ref{lem:moment_estimate} for some $C_4$ only depending on $C_1,C_2$ and $T$. So firstly choose $\gamma$ with $C_4/\gamma=\delta/4$, hence
\[P\left(\beta(T)\geq \gamma\right)\leq\frac{\delta}{4}.\]
Now choose $K$ with
\[P\left(\sup_{t\leq \gamma}|W'_{t}|\geq K\right) \leq\frac{\delta}{4},\; \;P(|X_0|\geq K)\leq\frac{\delta}{4},\;\;\frac{(C_1+C_2C_3)T}{K}\leq\frac{\delta}{4}.\]
Altogether, we obtain
\[P\left(\sup_{t\leq T}|X_t|\geq 3K\right)\leq \delta.\]
\end{proof}


For fixed $\rho , T>0$ define
$\tau_0=0$ and
\[
\tau_i=\inf\{s\geq \tau_{i-1}:|X_s-X_{\tau_{i-1}}|=\rho\}\wedge T, \;\;
n_T=\sup\{k:\tau_k<T\}.
\]
It is rather obvious that for a fixed process $Y$ it holds that $P(\tau_1=0)=0$ and $P(n_T<\infty)=1$. In the following Lemma, we show a version of this observation, that holds uniformly in all modulating processes $(Y_t)_{t\in[0,\infty)}$ taking values in $\Y$.

\begin{lem}\label{lem:N_t estimate}
In addition to the assumptions of Lemma \ref{lem:moment_estimate}, let us assume that
\begin{equation}\label{eq:bounded_coeff}
\sup_{|x| \leq K, y \in \Y} (|b(x,y)| +  \sigma^2(x,y)) < \infty\mbox{ for all $K>0$.}
\end{equation}
Then for any  $\rho , T, \delta > 0$ there exist $K', \rho' >0$, such that
\[P(\tau_1\geq \rho')\geq 1-\delta,\;\; P(n_T\leq K')\geq 1-\delta\]
for all processes $(Y_t)_{t\in[0,\infty)}$ taking values in $\Y$.
\end{lem}
\begin{proof} $(a)$ We start by looking at $\tau_1$. As in the proof of Lemma \ref{lem:unif} and using the same notation we have for $\rho' < T$
\begin{align*}
& P( \tau_1 \leq \rho') = P\left(\sup_{t\leq \rho'}|X_t - X_0| \geq \rho \right)\\
\leq&  \frac{(C_1 + C_2C_3)\rho'}{\rho/2} +      P(\beta(\rho') \geq \gamma)+P\left(\sup_{t\leq \gamma }| W'_t| \geq \frac{\rho}{2} \right)\\
\leq&  C'_1\frac{\rho'}{\rho} +      \frac{E\beta(\rho')}{\gamma}+P\left(\sup_{t\leq \gamma }| W'_t| \geq \frac{\rho}{2}\right)\\
\leq&  C'_1\frac{\rho'}{\rho} +     C'_2 \frac{\rho'}{\gamma}+P\left(\sup_{t\leq \gamma }| W'_t| \geq \frac{\rho}{2}\right).
\end{align*}
Firstly we choose $\gamma$ such that $P\left(\sup_{t\leq \gamma }| W'_t| \geq \rho/2\right) \leq \delta/3$.
Then we choose $\rho'$ such that $  C'_1\frac{\rho'}{\rho}\leq \delta/3,\;C'_2 \frac{\rho'}{\gamma}\leq \delta/3$,
which gives the first estimate of the assertion.

$(b)$ Let $\delta >0$. Choose $K$ according to Lemma \ref{lem:unif} such that
\[P\left(\sup_{t\leq T}|X_t|\geq K\right)\leq \frac{\delta}{2}.\]
We set \[ C = \sup_{|x| \leq K, y \in \Y} |b(x,y)|,\; D = \sup_{|x| \leq K, y \in \Y} \sigma^2(x,y).\]
The following estimates are always considered on $A= \{\sup_{t\leq T}|X_t| \leq K\}.$
Note that for any $\tau,s \geq 0$
\[ \left|\int_\tau^{\tau + s} b(X_s,Y_s)ds+\int_\tau^{\tau +s} \sigma(X_s,Y_s)dW_s\right| \geq \rho\]
implies
\[ \int_\tau^{\tau + s} |b(X_s,Y_s)| ds \geq \frac{\rho}{2} \mbox{ or } \left|\int_\tau^{\tau +s} \sigma(X_s,Y_s)dW_s\right| \geq \frac{\rho}{2}.\]
Hence on $A$
\[ \int_\tau^{\tau + s} C ds \geq \frac{\rho}{2} \mbox{ or } \left|\int_\tau^{\tau +s} \sigma(X_s,Y_s)dW_s\right| \geq \frac{\rho}{2} \]
so that
\[ \tau_i - \tau_{i-1} \geq \min\left\{\frac{\rho}{2C}, \;\inf\{s:
|\int_{\tau_{i-1}}^{\tau_{i-1} +s} \sigma(X_s,Y_s)dW_s| \geq \frac{\rho}{2}\}\right\}.\]
Now we may argue in the following way. If $\tau_k < T$ then there exist $k$ disjoint stochastic intervals
 $[ \tau_{i-1}, \tau_{i})= J_i \subseteq [0,T)$ such that
 \begin{center}
 the length of $J_i$ is $\geq$
$\frac{\rho}{2C}$ or $\sup_{s \in J_i}|\int_{\tau_{i-1}}^s \sigma(X_s,Y_s)dW_s| \geq \rho/2$.
\end{center}
The number $m$ of intervals $J_i$ with length $\geq \frac{\rho}{2C}\;$ must fulfil
$m\frac{\rho}{2C} < T$  so that there are at least
$$k-\frac{T2C}{\rho} \mbox{ intervals }J_i  \mbox{ with }
\sup_{s \in J_i}\left|\int_{\tau_{i-1}}^s \sigma(X_s,Y_s)dW_s\right| \geq \frac{\rho}{2}.$$
To obtain a bound independent of the particular process $(Y_t)_t$ we transfer this to the process $(W'_t)_t$. Looking at the intervals
$[\beta( \tau_{i-1}), \beta(\tau_{i}))= J'_i$ these are disjoint intervals $\subseteq
[0, \beta(T))$ and for at least $k-\frac{T2C}{\rho}$ of them we have
$$ \sup_{t \in J'_i} | W'_t - W'_{\beta(\tau_{i-1})}| \geq \frac{\rho}{2}.$$
Note that on $A$
$$ \beta (T) = \int_0^T\sigma^2(X_t,Y_t) dt \leq DT.$$
So it follows that $\tau_k < T$ implies the existence of at least $k-\frac{T2C}{\rho}$  such disjoint intervals $J'_i \subseteq [0,DT]$.

For a formal statement define the random variable
$$ Z_{\rho',T'} = \sup \{k: \mbox{ There exist $k$ disjoint intervals } \subseteq [0,T') \mbox{ with }
\sup_{a_i \leq t \leq b_i}|W'_t-W'_{a_i}|\geq \rho'\}.$$
By path continuity we see that
$$ P(Z_{\rho',T'} = \infty)=0, \mbox{ hence } P(Z_{\rho',T'} \geq k)\to 0
\mbox{ as } k \to \infty. $$
The foregoing reasoning implies
$$ P(n_T \geq k) = P(\tau_k < T) \leq P(A^c) + P\left(Z_{\rho/2,DT}\geq k-\frac{T2C}{\rho}\right).$$
So we only have to choose $K'$ such that
$P(Z_{\rho/2,DT}\geq K'-\frac{T2C}{\rho}) \leq \delta/2$ to obtain the second estimate. Note that $K'$ is independent of the particular process $(Y_t)_{t\in[0,\infty)}$.
\end{proof}

Before coming to the main result of this section, we provide an analytical lemma which is essential for the following.
\begin{lem}\label{lem:analytic}
Let $f:\R\rightarrow\R$ be measurable such that
\[\frac{1}{T}\int_0^Tf(x)dx\rightarrow0\mbox{ as }T\rightarrow\infty,\;\;\;\sup_T\frac{1}{T}\int_0^T|f(x)|dx<\infty.\]
Then for any continuous  $h:[0,1]\rightarrow\R$
\[\frac{1}{T}\int_0^Th\left(\frac{x}{T}\right)f(x)dx\rightarrow0\mbox{ as }T\rightarrow\infty.\]
\end{lem}
The proof can be found in the appendix.

Let us now fix some process $(Y_t)_{t\in[0,\infty)}$ with state space $\Y$. From now on we assume that 
 $\Y$ is a finite set, which implies that condition \eqref{eq:bounded_coeff} hold true. 
We assume that
$(Y_t)_{t\in[0,\infty)}$ is ergodic in the sense that
\[\frac{1}{t}\int_0^t(1_{\{Y_{s}=y\}}-\pi(y))ds\rightarrow0\mbox{ a.s., as }t\rightarrow\infty\]
for some probability distribution $\pi(y), y \in \Y.$
We increase the speed of the process by looking at
processes $(Y_t^\epsilon)_{t\in[0,\infty)}$, $\epsilon > 0$, having the same distribution as  $(Y_{t/\epsilon})_{t\in[0,\infty)}.$ The first processes have to be adapted to a filtration for which
 $(W_{t})_{t\in[0,\infty)}$ is a Wiener process, and we let  $(X_{t}^\epsilon)_{t\in[0,\infty)}$ be the solution of the corresponding stochastic differential equation
  \[dX_t^\epsilon=b(X_t^\epsilon,Y_t^\epsilon)dt+\sigma(X_t^\epsilon,Y_t^\epsilon)dW_t,\]
with starting value independent of $\epsilon.$ In our proof we will work with $Y_t^\epsilon= Y_{t/\epsilon}.$
For these processes to live on a common filtration with the Wiener process  $(W_{t})_{t\in[0,\infty)}$
we assume that
\[ (W_{t})_{t\in[0,\infty)} \mbox{ and } (Y_{t})_{t\in[0,\infty)} \mbox{ are independent}\]
and then we may use  $(Y_{t/\epsilon})_{t\in[0,\infty)}$ for  $(Y_t^\epsilon)_{t\in[0,\infty)}$. This is no restriction in generality when compared
with \cite{Skorohod}, \cite{Skorohod1}. When $(Y_{t})_{t\in[0,\infty)}$
is a Markov process with discrete state space living on the same filtration as
 $(W_{t})_{t\in[0,\infty)}$ then these two processes are necessarily independent.
 This is shown in \cite{Shreve} for a Poisson process and can be generalized to general Markov processes with discrete state space; see e.g. \cite{Christensen}.

\begin{satz}\label{satz} Let $B:I\times \Y\rightarrow\R $ be such that $B(\cdot,y)$ is continuous for all $y \in \Y.$ Then for all $T>0,y\in \Y$ it holds that
\[\sup_{0\leq r\leq T}\left|\int_0^rB(X_t^\epsilon,y)\left(1_{\{Y_t^\epsilon=y\}}-\pi(y)\right)dt\right|\rightarrow0\mbox{ in probability as }\epsilon\rightarrow0.\]
\end{satz}
\begin{proof} Fix $T>0,\;y\in \Y$. Let $\eta>0$. We want to show that
\[P\left(\sup_{r\leq T}\left|\int_0^rB(X_t^\epsilon,y)\left(1_{\{Y_t^\epsilon=y\}}-\pi(y)\right)dt\right|>\eta\right)\rightarrow0\mbox{ as }\epsilon\rightarrow0.\]
Let $\delta>0$. Choose $K$ according to Lemma \ref{lem:unif} such that
\[P\left(\sup_{t\leq T}|X_t^\epsilon|\geq K\right)\leq \delta\mbox{ for all }\epsilon>0.\]
Let
\[A_\epsilon=\left\{\sup_{t\leq T}|X_t^\epsilon|<K\right\}.\]
The following estimates are always considered on $A_\epsilon$. Let $$C=\sup\{|B(x,y)|:|x|\leq K\}<\infty;$$ thus for $r_0=\eta/C$
\[\sup_{r\leq r_0}\int_0^r\left|B(X_t^\epsilon,y)\left(1_{\{Y_t^\epsilon=y\}}-\pi(y)\right)\right|dt\leq \eta.\]
By a change of variable, we obtain
\begin{align*}
\int_0^rB(X_t^\epsilon,y)(1_{\{Y_t^\epsilon=y\}}-\pi(y))dt=&\int_0^rB(X_t^\epsilon,y)(1_{\{Y_{t/\epsilon}=y\}}-\pi(y))dt\\
=&{\epsilon}\int_0^{r/\epsilon}B(X_{\epsilon s}^\epsilon,y)(1_{\{Y_{s}=y\}}-\pi(y))ds.
\end{align*}
Note that the integral
\begin{align*}
\frac{\epsilon}{r}\int_0^{r/\epsilon}B(X_{rs \epsilon/r}^\epsilon,y)(1_{\{Y_{s}=y\}}-\pi(y))ds
\end{align*}
has exactly the form considered in the proof of Lemma \ref{lem:analytic} (see Appendix \ref{appendixa}) with $t=r/\epsilon$, $g=B(\cdot,y),\;h(s)=X_{rs}^\epsilon,$ and $f(s)=1_{\{Y_s=y\}}-\pi(y)$. On $A_\epsilon$ one has
\[|X_{rs}^\epsilon|\leq K \mbox{ for all }s\leq 1.\]
Let $\delta>0$. Since $B(\cdot,y)$\ is uniformly continuous on $[-K,K]$ we may choose $\rho>0$ such that
\[|B(x,y)-B(x',y)|\leq\frac{\eta}{2T}\;\;\;\;\mbox{ for }|x-x'|\leq \rho,\;x,x'\in[-K,K].\]
Next, let $\sigma_0^{\epsilon,r}=0$,
\begin{align*}
&\sigma_i^{\epsilon,r}=\inf\{s\geq \sigma_{i-1}^{\epsilon,r}:|X_{rs}^\epsilon-X^\epsilon_{\sigma_{i-1}^{\epsilon,r}}|=\rho\}\wedge 1,
&n^{\epsilon,r}=\sup\{i:\sigma_i^{\epsilon,r}<1\}.
\end{align*}
The estimate in in the proof of Lemma \ref{lem:analytic} with $C_1=1,C_2=C$ yields
\begin{align*}
\left|\epsilon\int_0^{r/\epsilon}B(X_{rs\epsilon/r}^\epsilon,y)(1_{\{Y_{s}=y\}}-\pi(y))ds\right|
\leq r\frac{\eta}{2}+2rC\sum_{i=1}^{n^{\epsilon,r}+1}\left|\epsilon\int_0^{\sigma_i^{\epsilon,r}r/\epsilon}(1_{\{Y_{s}=y\}}-\pi(y))dy\right|.
\end{align*}
Setting $\tau_0^\epsilon=0$,
\begin{align*}
&\tau_i^\epsilon=\inf\{s\geq \tau_{i-1}^\epsilon:|X_s^\epsilon-X^\epsilon_{\tau_{i-1}^\epsilon}|=\rho\}\wedge T,
&n^\epsilon=\sup\{k:\tau_k^\epsilon<T\},
\end{align*}
it follows that
$r\sigma_i^{\epsilon,r}=\tau_i^\epsilon\wedge r,\;\;n^{\epsilon,r}\leq n^\epsilon$ and
\begin{align*}
\left|\epsilon\int_0^{r/\epsilon}B(X_{\epsilon s}^\epsilon,y)(1_{\{Y_{s}=y\}}-\pi(y))ds\right|
\leq\frac{\eta}{2}+2C\sum_{i=1}^{n^{\epsilon,r}+1}\epsilon\int_{0}^{\tau_i^\epsilon\wedge r/\epsilon}(1_{\{Y_{s}=y\}}-\pi(y))ds.
\end{align*}
Next, according to Lemma \ref{lem:N_t estimate}, we may choose $K',\rho'\leq r_0$ such that
\[P(n^\epsilon\leq K')\geq 1-\delta,\;\;P(\tau_1^\epsilon\geq \rho')\geq 1-\delta\mbox{ for all }\epsilon>0,\]
hence also
\[P(\tau_1^\epsilon\wedge r\geq \rho')\geq 1-\delta\mbox{ for all }\epsilon>0,\;r\geq r_0.\]
Due to the ergodicity assumption
\[\frac{1}{t}\int_0^t(1_{\{Y_{s}=y\}}-\pi(y))ds\rightarrow0\mbox{ a.s., as }t\rightarrow\infty.\]
So we may choose $\epsilon_0>0$ such that for
\[D_\epsilon=\left\{\sup_{t\geq \rho'}\big|\epsilon\int_0^{t/\epsilon}(1_{\{Y_{s}=y\}}-\pi(y))ds\big|\leq \frac{\eta}{2}\frac{1}{2(K'+1)C}\right\}\]
we have $P(D_\epsilon)\geq 1-\delta\mbox{ for all }\epsilon\leq \epsilon_0$. \\ Altogether, we obtain on $A_\epsilon\cap\{n^\epsilon\leq K'\}\cap\{\tau_1^\epsilon\geq \rho'\}\cap D_\epsilon$ that for all $r_0\leq r\leq T,\;0<\epsilon\leq \epsilon_0$
\begin{align*}
\left|\epsilon\int_0^{r/\epsilon}B(X_{\epsilon s}^\epsilon,y)(1_{\{Y_{s}^\epsilon =y\}}-\pi(y))ds\right|
\leq\frac{\eta}{2}+2C(n^\epsilon+1)\frac{\eta}{2}\frac{1}{2(K'+1)C}\leq	\eta,
\end{align*}
hence
\begin{align*}
\sup_{0<r\leq T}\left|\int_0^{r}B(X_{s}^\epsilon,y)(1_{\{Y_{s}^\epsilon =y\}}-\pi(y))ds\right|\leq	 \eta,
\end{align*}
the case $r \leq r_0$ being obvious as remarked in the beginning of the proof. It follows that
\begin{align*}
&P\left(\sup_{0<r\leq T}\left|\int_0^{r}B(X_{s}^\epsilon,y)(1_{\{Y_{s}^\epsilon =y\}}-\pi(y))ds\right|>\eta\right)\\
&\leq P(A_\epsilon^c)+P(n^\epsilon\geq K')+P(\tau_1^\epsilon\leq \rho')+P( D_\epsilon^c)\leq 4\delta,
\end{align*}
for all $0<\epsilon\leq \epsilon_0$.
\end{proof}
Now, the previous result can be utilized to prove the convergence in distribution for fast switching diffusions as follows.\\
Define \begin{eqnarray*} &&\hat{b} : I \rightarrow \R,\,\; \hat{b} (x)= \sum_{y \in \Y}b(x,y)\pi(y),\\
 &&\hat{\sigma} : I \rightarrow \R,\,\; \hat{\sigma} (x)= \sum_{y \in \Y}\sigma(x,y)\pi(y),
\end{eqnarray*}
and
$(\hat{X}_{t})_{t\in[0,\infty)}$ as the solution of the corresponding stochastic differential equation
  \[d\hat{X}_t=\hat{b}(\hat{X}_t)dt+\hat{\sigma}(\hat{X}_t)dW_t.\]

\begin{satz}\label{satz2}
\[ (X^\epsilon_{t})_{t\in[0,\infty)} \rightarrow (\hat{X}_{t})_{t\in[0,\infty)}
\mbox{ in distribution as }\epsilon \rightarrow 0.\]
\end{satz}
\begin{proof}
The infinitesimal characteristics are
\[ b(X_t^\epsilon, Y_t^\epsilon),\; \sigma(X_t^\epsilon, Y_t^\epsilon) \mbox{ for the semimartingale }  (X^\epsilon_{t})_{t\in[0,\infty)},\]
and
\[ \hat{b}(\hat{X}_t),\; \hat{\sigma}(\hat{X}_t)\mbox{ for the semimartingale }
(\hat{X}_{t})_{t\in[0,\infty)}.\]
Theorem \ref{satz} shows that for all $T>0$
\[\sup_{0\leq r\leq T}\int_0^r \left( b(X_t^\epsilon,Y_t^\epsilon) -\hat{b}(X_t^\epsilon)\right)dt\rightarrow 0\mbox{ in probability as }\epsilon\rightarrow0,\]
and similarly for $\sigma, \hat{\sigma}.$
This implies the assertion by the semimartingale convergence  theorem; see \cite[Theorem 3.21, Chapter IX]{JacodShiryaev}.
\end{proof}

\section{Discrete state processes}\label{sec:discrete}
In the case of a discrete state space we start with a c\`adl\`ag process $(Y_t)_{t\in[0,\infty)}$ with a finite state space $\Y$, assumed a subset of $\R$ without loss of generality, and discrete jump times $\gamma_0=0<\gamma_1<\gamma_2<...$. Here, the term discrete  jump times means that these times are strictly increasing and the process $Y$ is constant on $[\gamma_{i+1},\gamma_i)$ for all $i$. Furthermore, let $I\subseteq \R$ be countable, and for each $y\in \Y$ let $q(\cdot,\cdot|y)$ be an intensity matrix. Conditionally, we generate the switching process in the following way:  In zero, we start a continuous time Markov chain $\tilde X^0$ with starting state $x_0$ and intensity matrix $q(\cdot,\cdot|Y_0)$. In the first jump time $\gamma_1$, we start a new chain $\tilde X^1$ with starting point $\tilde X^0_{\gamma_1}$ and intensity matrix $q(\cdot,\cdot|Y_{\gamma_1})$. In the same way we define $\tilde X^{i+1}$, starting in $\tilde X^{i}_{\gamma_i}$. We define the switching process $X=X^Y$ by setting  $X_t=\tilde X^i_{t-\gamma_i}$ for $\gamma_i\leq t<\gamma_{i+1}$.\\
Note that when using diffusion processes $\tilde X^i$\ with coefficients depending on $y$ instead of Markov chains in the previous construction, (ignoring some technical issues) the process $X$ is a switching diffusion as considered in Section \ref{sec:diffusion}. In this section we consider the jump counterpart for the results obtained there.

We make the assumption that for some finite set  $J\subseteq \R\setminus \{0\}$
\[q(i,i+j|y)=0\mbox{ for all }i\in I,\;y\in \Y,\;j\not\in J,\]
so $i+J$ is the set of states which can be reached from $i$.
Thus $q(i|y):=\sum_{i'\not=i}q(i,i'|y)<\infty$ for all $y\in\Y$, and we furthermore assume that
\[q:=\sup_{i,y}q(i|y)<\infty.\]
If the state space $I$ of the process $X$ is finite, these assumptions are of course fulfilled, and they imply that there is no explosion in finite time.\\
Define the jumps time of $(X_t)_{t\geq 0}$ by $\tau_0=0$ and
\[\tau_i=\inf\{t\geq \tau_{i-1}:X_t\not=X_{\tau_{i-1}}\}.\]
The following lemma is the analogon to Lemma \ref{lem:N_t estimate} and gives again an estimate, which is uniform in all modulating processes $Y$. The proof now turns out to be much easier due to the discrete nature of situation.
\begin{lem}\label{lem:asmpt_discrete}
Let\ $T>0$, $n_T=\sup\{k:\tau_k<T\}$. Let\ $\delta>0$. Then there exists $\rho'>0,\;K'\in\N$ such that
\[P(\tau_1\geq \rho')\geq 1-\delta,\;\;P(n_T<K')\geq 1-\delta\] for all processes $(Y_t)_{t\in[0,\infty)}$ taking values in $\Y$.
\end{lem}

\begin{proof}
Denote the counting process of $(X_t)_{t\in[0,\infty)}$ by $(N_t)_{t\in[0,\infty)}$, and by $(N_t^i)_{t\in[0,\infty)}$ for $(X_t^i)_{t\in[0,\infty)}$, i.e. $N_t$ (resp. $N_t^i$) denotes the number of jumps of $X$ (resp. $X^i$) before time $t$. Let the random index $j_t$ be given by $\gamma_{j_t}\leq t<\gamma_{j_t+1}$. Then we consider
\[\{\tau_1>t\}=\{N_t=0\}=\{N_{\gamma_1}^0=0\}\cap \{N_{\gamma_2-\gamma_1}^1=0\}\cap...\cap \{N_{t-\gamma_{j_t}}^{j_t}=0\}.\]
It follows by conditioning that
\begin{align*}
P(N_t=0)&=E\left(e^{-q(X_0|Y_0)\gamma_1}e^{-q(X_{\gamma_1}|Y_{\gamma_1})({\gamma_2}-\gamma_1)}...e^{-q(X_{\gamma_{j_t}}|Y_{\gamma_{j_t}})(t-\gamma_{j_t})}\right)\geq e^{-qt}.
\end{align*}
Now, let $(N^*_t)_{t\in[0,\infty)}$ denote a Poisson process with intensity $q$ and corresponding jump times $\tau_j^*,j\geq 0.$ Then the previous arguments show that $P(\tau_1>t)\geq P(\tau_1^*>t)$. Together with the Markov property, the same arguments can be used for $\tau_j,\tau_j^*,\;j>1$, so that $P(\tau_j>t)\geq P(\tau_j^*>t)$. Hence
\begin{align*}
P(\tau_1\geq \rho)\geq P(\tau_1^*\geq \rho)\rightarrow1&\mbox{ as }\rho\rightarrow0,\\
P(n_T\leq K)=P(\tau_{k+1}>T)\geq P(\tau_{K+1}^*>T)\rightarrow 1&\mbox{ as }K\rightarrow\infty.
\end{align*}
\end{proof}

\begin{lem}\label{cor:unif_estimate_discrete}
Assume the situation of Lemma \ref{lem:asmpt_discrete}. For any $\epsilon>0$ there exists $K$  such that
\[P\left(\sup_{t\leq T}|X_t|\geq K\right)\leq \epsilon\] 
uniformly in all $\Y$-valued modulating processes $(Y_t)_{t\in[0,\infty)}$ for $X=X^Y$.
\end{lem}
\begin{proof}
Using Markov's inequality we have
\[P\left(\sup_{t\leq T}|X_t|\geq K\right)\leq E(\sup_{t\leq T}|X_t|)/K\leq (|x_0|+E(n_T+1)\sup_{j\in J}|j|)/K.\]
Therefore, it is enough to show that $E(n_T)$ is bounded uniformly in $Y$. But this holds by the proof of the preceding Lemma since $En_T$ is not larger than the expected number of jumps of the Poisson process $(N^*_t)_{t\in[0,\infty)}$ in $[0,T]$, which is well-known to be finite.
\end{proof}

As in Section \ref{sec:diffusion} we fix a process $(Y_t)_{t\in[0,\infty)}$ with the property
\[\frac{1}{t}\int_0^t(1_{\{Y_{s}=y\}}-\pi(y))ds\rightarrow0\mbox{ a.s., as }t\rightarrow\infty.\]
Furthermore, we take $(Y^\epsilon_t)_{t\in[0,\infty)} = (Y_{t/\epsilon})_{t\in[0,\infty)}$ with corresponding  $(X^\epsilon_t)_{t\in[0,\infty)}$.


\begin{satz}  Let $Q:I\times \Y\rightarrow\R$ be such that
$\sup_{|i|\leq K}|Q(i,y)|<\infty\mbox{ for all }y\in\Y,\;K>0.$
For all $y\in\Y,\,T>0$ it holds that
\[\sup_{0\leq r\leq T}\left|\int_0^rQ(X_t^\epsilon,y)\left(1_{\{Y_t^\epsilon=y\}}-\pi(y)\right)dt\right|\rightarrow0\mbox{ in probability as }\epsilon\rightarrow0.\]
\end{satz}
\begin{proof}
Fix $y,T$. For $\delta>0$ choose $K$ according to Lemma \ref{cor:unif_estimate_discrete} such that
\[P\left(\sup_{t\leq T}|X_t^\epsilon|\geq K\right)\leq \delta \mbox{ for all }\epsilon>0.\]
Write
\[A_\epsilon=\{\sup_{t\leq T}|X_t^\epsilon|\leq K\},\;\;C=\sup\{|Q(i,y)|:|i|\leq K\}<\infty.\]
Then for $r_0=\eta/C$ on $A_\epsilon$
\[\sup_{r\leq r_0}\int_0^r|Q(X_t^\epsilon,y)(1_{\{Y_t=y\}}-\pi(y))|dt\leq \eta.\]
We may proceed with a simplified version of the proof of Theorem \ref{satz} without the use of Lemma \ref{lem:analytic}. It holds that
\begin{align*}
\int_0^rQ(X_t^\epsilon,y)(1_{\{Y_t^\epsilon=y\}}-\pi(y))dt={\epsilon}\int_0^{r/\epsilon}Q(X_{rs\epsilon/r }^\epsilon,y)(1_{\{Y_{s}=y\}}-\pi(y))ds.
\end{align*}
Define $\sigma_i^{\epsilon,r}$ as in the proof of Theorem \ref{satz} replacing $=\rho$ by $>0$ to obtain the jump times, with corresponding $n^{\epsilon,r}$. Then, on $A_\epsilon$
\begin{align*}
\left|{\epsilon}\int_0^{r/\epsilon}Q(X_{rs\epsilon/r }^\epsilon,y)(1_{\{Y_{s}^\epsilon =y\}}-\pi(y))ds\right|&\leq \left|{\epsilon}\sum_{i=1}^{n^{\epsilon,r}+1}Q(X_{r\sigma_i^{\epsilon,r}}^\epsilon,y)\int_{\sigma_{i-1}^{\epsilon,r}r/\epsilon}^{\sigma_{i}^{\epsilon,r}r/\epsilon}(1_{\{Y_{s}=y\}}-\pi(y))ds\right|\\
&\leq 2C(n^{\epsilon,r}+1)\left|{\epsilon}\int_{0}^{\sigma_{n^{\epsilon,r}+1}^{\epsilon,r}r/\epsilon}(1_{\{Y_{s}=y\}}-\pi(y))ds\right|
\end{align*}
Using Lemma \ref{lem:asmpt_discrete}, the proof is concluded as in Theorem \ref{satz}.
\end{proof}

Define the intensity matrix
\[ \hat{q}(\cdot,\cdot) = \sum_{y \in \Y}q(\cdot,\cdot|y)\pi(y)\]
with corresponding Markov process $(\hat{X}_t)_{t\in[0,\infty)}$.

\begin{satz}
\[ (X^\epsilon_t)_{t\in[0,\infty)} \rightarrow (\hat{X}_{t})_{t\in[0,\infty)}
\mbox{ in distribution as }\epsilon \rightarrow 0.\]
\end{satz}
\begin{proof}
The infinitesimal jump characteristics are given by
\[ q(X_t^\epsilon, X_t^\epsilon +j | Y_t^\epsilon)\mbox{ for the semimartingale }  (X^\epsilon_{t})_{t\in[0,\infty)},\]
and
\[ \hat{q}(\hat{X}_t,\hat{X}_t + j)\mbox{ for the semimartingale }
(\hat{X}_{t})_{t\in[0,\infty)}.\]
Theorem \ref{satz} shows that for all $T>0,\; j \in J$
\[\sup_{0\leq r\leq T}\int_0^r \left( q(X_t^\epsilon,X_t^\epsilon + j|Y_t^\epsilon) -\hat{q}(X_t^\epsilon,X_t^\epsilon + j)\right)dt\rightarrow 0\mbox{ in probability as }\epsilon\rightarrow 0\]
using $Q(i,y)=q(i,i+j|y)$. Again
this implies the assertion by the semimartingale convergence  theorem; see \cite[Theorem 3.21, Chapter IX]{JacodShiryaev}.
\end{proof}

\section*{Acknowledgements}
We thank the anonymous referee for carefully reading an earlier version of this manuscript and for thoughtful and valuable comments that helped to improve the presentation of this article.

\begin{appendix}
\section{Proof of Lemma \ref{lem:analytic}}\label{appendixa}

\begin{proof}[Proof of Lemma \ref{lem:analytic}]
Note that in the proof we shall provide a more precise inequality which will be used in proving Theorem \ref{satz}. Due to this reason we shall use a further continuous mapping $g:\R\rightarrow\R$ and write $g(h({x}/{T}))$ instead of $h({x}/{T})$. For the assertion of this analytical lemma, $g$ is just the identity. Let $[\alpha,\beta]=h([0,1])$. Let $\delta>0$.\ Since $g$ is uniformly continuous on $[\alpha,\beta]$ there exists $\rho>0$ such that
\[|g(y)-g(y')|\leq \frac {\delta}{C_1} \mbox{ for }|y-y'|\leq \rho,y,y'\in[\alpha,\beta],\]
where $C_1=\sup_T\frac{1}{T}\int_0^T|f(x)|dx$; also let $C_2=\sup_{y\in[\alpha,\beta]}|g(y)|.$ Set
\[s_0=0,\;\;s_i=\inf\{s\geq s_{i-1}:|h(s)-h(s_{i-1})|=\rho\}\wedge 1,\]
furthermore $n=\sup\{i:s_i<1\}$, where, due to the continuity of $h$, $n$ is finite. Now
\begin{align*}
&\left|\frac{1}{t}\int_0^tg\left(h\left(x/t\right)\right)f(x)dx\right|\\
\leq&\left|\frac{1}{t}\sum_{i=1}^{n+1}\int_{ts_{i-1}}^{ts_i}\left(g\left(h\left(x/t\right)\right)-g\left(h\left(s_i\right)\right)\right)f(x)dx\right|+\left|\frac{1}{t}\sum_{i=1}^{n+1}\int_{ts_{i-1}}^{ts_i}g\left(h\left(s_i\right)\right)f(x)dx\right|\\
\leq &\frac{1}{t}\frac{\delta}{C_1}\sum_{i=1}^{n+1}\int_{ts_{i-1}}^{ts_i}|f(x)|dx+\frac{1}{t}\sum_{i=1}^{n+1}\left|g\left(h\left(s_i\right)\right)\right|\left|\int_{ts_{i-1}}^{ts_i}f(x)dx\right|\\
\leq & \frac{\delta}{C_1}\frac{1}{t}\int_0^t|f(x)|dx+\frac{1}{t}\sum_{i=1}^{n+1}C_2\left|\int_{0}^{ts_i}f(x)dx-\int_{0}^{ts_{i-1}}f(x)dx\right|\\
\leq & \delta+2C_2\sum_{i=1}^{n+1}\left|\frac{1}{t}\int_{0}^{ts_i}f(x)dx\right|.
\end{align*}
Now choose $t_0$ such that
\[\sup_{s\geq s_1}\left|\frac{1}{t}\int_0^{ts}f(x)dx\right|\leq \frac{\delta}{2C_2(n+1)}\mbox{ for all }t\geq t_0,\]
thus
\[\left|\frac{1}{t}\int_0^tg\left(h\left(x/t\right)\right)f(x)dx\right|\leq \delta+2C_2(n+1)\frac{\delta}{2C_2(n+1)}=2\delta.\]
\end{proof}
\end{appendix}

\newpage
\addcontentsline{toc}{section}{Literature}
\bibliographystyle{plainnat}
\bibliography{literaturefastswitching}


\end{document}